\documentclass[11pt,a4paper]{amsart}
\usepackage[latin1]{inputenc}
\usepackage{a4wide}

\usepackage{amsmath,amssymb,amsthm}
\usepackage{graphicx}
\usepackage{color}
\usepackage{caption,subcaption}
\newcommand{\N}{\mathbb{N}}
\newcommand{\Z}{\mathbb{Z}}

\newcommand{\R}{\mathbb{R}}
\newcommand{\rr}{\mathbb{R}}

\newcommand{\dx}{{\Delta x}}
\newcommand{\dt}{{\Delta t}}

\newcommand{\eps}{\varepsilon}

\newtheorem{theorem}{Theorem}[section]
\newtheorem{prop}{Proposition}[section]

\newtheorem{lemma}{Lemma}[section]
\newtheorem{remark}{Remark}

\numberwithin{equation}{section}

\def\uud{u^\Delta}

\def\ul{u^\lambda}
\def\psidt{\psi^{\dt}}
\def\psil{\psi^{\frac{\dt}{\lambda^2}}}
\def\R{\mathbb{R}}
\def\R{\mathbb{R}}
\newcommand{\sgn}{\,{\rm sgn}}

\def\la{\langle}
\def\ra{\rangle}

\title[Splitting method for Augmented Burgers]{A splitting method for the augmented Burgers equation}
\author[L. I. Ignat, A. Pozo]{L. I. Ignat, A. Pozo}
\date{\footnotesize \today}

\address{Liviu I. Ignat
\hfill\break\indent Institute of Mathematics ``Simion Stoilow'' of the Romanian Academy\\
\hfill\break\indent 21 Calea Grivitei Street \\010702 Bucharest \\ Romania.}
 \email{{\tt liviu.ignat@gmail.com}\hfill\break\indent {\it Web page: }{\tt http://www.imar.ro/\~\,lignat}}
\address{Alejandro Pozo
\hfill\break\indent CBT - Innovalia\\
\hfill\break\indent Carretera de Asua 6, E-48930 Las Arenas - Getxo, Basque Country - Spain.
\hfill\break\indent \and
\hfill\break\indent BCAM - Basque Center for Applied Mathematics\\
\hfill\break\indent Alameda de Mazarredo 14, E-48009 Bilbao, Basque Country - Spain.}
 \email{{\tt alejandropozo@gmail.com}}

\begin{document}
\begin{abstract}
	In this paper we consider a splitting method for the augmented Burgers equation and prove that it is of first order. We also analyze the large-time behavior of the approximated solution by obtaining the first term in the asymptotic expansion. We prove that, when time increases, these solutions behave as the self-similar solutions of the viscous Burgers equation. 
\end{abstract}

\maketitle

%%%%%%%%%%%% SECTION 1 %%%%%%%%%%%%%%%%%%%%%%%%%
\section{Introduction}

In this paper we consider a splitting method for the following augmented Burgers equation:
\begin{equation}\label{eq:abe}
	\begin{cases}
		u_t-\left(\frac{u^2}{2}\right)_x=u_{xx}+K*u-u+u_x,&(t,x)\in(0,\infty)\times\R, \\
		u(0,x)=u_0(x),&x\in\R,
	\end{cases}
\end{equation}
where $K(z)=e^{-z}\chi_{(0,\infty)}$. This model has been introduced in the context of the analysis of the sonic-boom phenomenon to reproduce the propagation of the sound waves produced by supersonic aircrafts from their near-field down to the ground level \cite{Cleveland:1995, Rallabhandi:2011a, Rallabhandi:2011b}.

New trends in aerospace engineering have started to use nonlinear physical models to characterize this phenomenon, in order to be capable of building efficient supersonic civilian airplanes that can avoid such strident noise. We refer the reader to \cite{Alonso:2012} for a detailed survey on the topic. In this paper we focus on the augmented Burgers equation, which was initially developed by Cleveland \cite{Cleveland:1995} and later adopted by Rallabhandi \cite{Rallabhandi:2011a, Rallabhandi:2011b}. 

Let us now explain the general setting from where model \eqref{eq:abe} appears. The nonlinear equation that models the sonic-boom phenomenon is given by
\begin{equation}\label{eq:sonic}
	\frac{\partial P}{\partial \sigma} = P \frac{\partial P}{\partial \tau} + \frac{1}{\Gamma}\frac{\partial^2 P}{\partial \tau^2}+\sum_{\nu} C_\nu\frac{1}{1+\theta_\nu\frac{\partial}{\partial \tau}}\frac{\partial^2P}{\partial\tau^2}-\frac{1}{2G}\frac{\partial G}{\partial\sigma}P+\frac{1}{2\rho_0c_0}\frac{\partial(\rho_0c_0)}{\partial\sigma} P,
\end{equation}
where $P=P(\sigma,\tau)$ is the dimensionless perturbation of the atmosphere pressure distribution. The covered distance $\sigma$ and time of the perturbation $\tau$ are also dimensionless. The operator appearing in the summation, corresponding to the molecular relaxations,  is defined by:
\begin{equation}\label{eq:operator}
	\frac{1}{1+\theta\frac{\partial}{\partial\tau}} f(\tau) = \frac{1}{\theta}\int_{-\infty}^\tau e^{(\xi-\tau)/\theta} f(\xi) d\xi = K_{\theta}*f(\tau),
\end{equation}
where $\theta K_{\theta}(\theta\ \cdot) = K(\cdot)$. Typically, two relaxation modes are considered (one for Oxygen molecules and another one for Nitrogen molecules), each of them with their corresponding dimensionless relaxation time $\theta_\nu$ and dispersion parameter $C_\nu$. $\Gamma$ is a dimensionless thermo-viscous parameter and function $G\equiv G(\sigma)$ denotes the ray-tube area. Density $\rho_0\equiv\rho(\sigma)$ and speed of sound $c_0\equiv c_0(\sigma)$, both closely related to the altitude of the flight, determine the atmosphere conditions. We refer to \cite{Cleveland:1995} for the complete formulation and setting-up of this model. In this paper we analyze the particular model \eqref{eq:abe} that can be obtained from \eqref{eq:sonic} by choosing $p_0$, $c_0$ to be constant, $ \Gamma=1$ and a single relaxation phenomenon with $c_\nu=\theta_\nu=1$.

Industrial applications like the aforementioned sonic-boom phenomenon, in which approximate solutions for large time are required, need a careful numerical treatment. A good understanding of the behavior of the solutions in these extended regimes, both in continuous and discrete settings, is necessary in order to be able to simulate them accurately. In that sense, equation \eqref{eq:abe} has already been analyzed in \cite{Pozo:2014}. The authors have obtain the asymptotic profile of the solutions of \eqref{eq:abe} as  time goes to infinity. Moreover, they propose a semi-discrete scheme that is capable of mimicking the large-time dynamics of the continuous system. This issue needs to be treated carefully, since numerical schemes with an acceptable accuracy in short-time intervals can completely disturb the large-time behavior of solutions when the numerical viscosity dominates the physical one \cite{Pozo:2015}. However, solving \eqref{eq:abe} by discretizing the convolution can be computationally expensive. In this paper we set up the framework to empower the use of splitting methods to solve the equation more efficiently. Let us mention that these techniques have already been used in \cite{Cleveland:1995,Rallabhandi:2011a,Rallabhandi:2011b} in the context of the sonic-boom phenomenon. Various versions of this method have also been developed, for instance, for the nonlinear Schr\"odinger, Korteweg-de Vries, Boltzmann... (see \cite{Holden:2010} and the references therein) and, more recently, for the nonlocal Fowler equation \cite{Bouharguane:2013}.

The basic idea behind operator splitting methods is that the overall evolution operator can be formally written as a sum of evolution operators for each term appearing in the model. In other words, one separates the complete system into a set of simpler sub-equations, for which more practical algorithms are available. Once we successively solve  these sub-equations, their solutions are put together to compute the complete solution of the model. We refer to \cite{Holden:2010} for a more detailed introduction on operator splitting methods.

Let us remark that the analysis performed in this paper can be extended to any kernel $K\in L^1(\R,1+|x|^2)$ by replacing the term $K\ast u-u+u_x$ adequately. The analysis we present here for $K(z)=e^{-z}\chi_{(0,\infty)}$  
is the same  for any kernel having the mass and the first moment equal to one:
\begin{equation}\label{cond.K}
	\int _{\R}K(z)dz=\int_{\R}K(z)zdz=1.
\end{equation}

We introduce the following Trotter formula for the augmented Burgers equation \eqref{eq:abe}. Let $X^t$ be the evolution operator associated with
\begin{equation}\label{eq:augmented}
	\begin{cases}
		v_t=K*v-v+v_x,&(t,x)\in(0,\infty)\times\R, \\
		v(t=0,x)=v_0(x),&x\in\R,
	\end{cases}
\end{equation}
and $Y^t$ the one corresponding to
\begin{equation}\label{eq:burgers}
	\begin{cases}
		w_t-\left(\frac{w^2}{2}\right)_x=w_{xx},&(t,x)\in(0,\infty)\times\R, \\
		w(t=0,x)=w_0(x),&x\in\R.
	\end{cases}
\end{equation}
We consider the flow $Z^t$ defined by
\begin{equation}\label{eq:zt} 
	Z^t=X^tY^t.
\end{equation}
The aim is to approximate the flow $S^t$ given by \eqref{eq:abe} by alternating both operators; that is, for small $\dt>0$ and positive integers $n$ we want to analyze whether 
\begin{equation*}
	S^{n\dt}u_0 =u(n\dt) \approx (Z^\dt)^n u_0=(X^\dt Y^\dt)^nu_0.
\end{equation*}
Subsequently we denote $Z^{n\dt}$ instead of $(Z^\dt)^n$ for the sake of clarity of the notation.

The first result of this paper is given in the following theorem. It confirms that the splitting \eqref{eq:zt} is first-order accurate for $H^2(\rr)$-solutions. For practical issues we also consider the case of $H^1(\rr)$-solutions. 
%Similar arguments can be done for $H^s(\rr)$ solutions with $s>1/2$ but this is left to the interested reader.

\begin{theorem} \label{teo:order}
	Let $r\in \{1,2\}$. For any $u_0\in {H^r(\R)}$ and for all $T>0$, the following hold:
		\begin{equation*}
		\left\|Z^{n\dt}u_0-u(n\dt)\right\|_{L^2(\R)}\le C(T,\|u_0\|_{H^1}) \|u_0\|_{H^r(\rr)}^2 (\dt)^{r/2}
	\end{equation*}
	and
	\begin{equation*}
		\left\|Z^{n\dt}u_0\right\|_{H^r(\R)}\le C(T,\|u_0\|_{L^2(\rr)}) \|u_0\|_{H^r(\rr)},
	\end{equation*}
		 for all $\dt\in(0,1)$ and for all $n\in\N$ such that $0\le n\dt\le T$.
\end{theorem}

The second result of this paper concerns the first term in the asymptotic expansion of the approximate solution $Z^{n\Delta t}u_0$. To be more precise let us define $t_n=n\dt$ and $t_{n+1/2}=(n+\frac12)\dt$ for every $n\in\N\cup\{0\}$ and the function $\uud$ by
\begin{equation}\label{u.delta}
	\uud(t,x)=
	\begin{cases}
		Y^{2(t-t_n)}Z^{n\dt}u_0(x),& t\in[t_n, t_{n+1/2}),\ x\in\R, \\
		X^{2(t-t_{n+1/2})}Y^{\dt}Z^{n\dt}u_0(x),& t\in[t_{n+1/2},t_{n+1}) ,\ x\in\R .
	\end{cases}
\end{equation}
Observe that for any $n\geq 0$, $\uud(t_n)=Z^{n\dt}u_0$. By construction and the properties of $X^t$ and $Y^t$, we will have that for any $u_0\in L^p(\R)$, $1\leq p<\infty$, $\uud$ defined by \eqref{u.delta} satisfies $\uud\in C([0,\infty);L^p(\R))$ for any $1\leq p\leq \infty$. The behavior of $\uud$ as $t\to\infty$, so the one of $Z^{n\dt}u_0$, follows from a scaling argument. Indeed, the following result confirms that the Trotter formula \eqref{eq:zt} for the augmented Burgers equation \eqref{eq:abe} preserves the large-time behavior obtained in \cite{Pozo:2014}.

\begin{theorem}\label{teo:asymptotic}
	For any $u_0\in L^1(\R)\cap L^\infty(\R)$ and for all $p\in[1,{\infty)}$, it holds that
	\begin{equation}\label{eq:convrates}
		t^{\frac12(1-\frac1p)}\left\|\uud(t)-u_M(t)\right\|_{L^p(\R)}\to\infty,\quad\mbox{as }t\to\infty,
	\end{equation}
	where $u_M(t,x)=t^{-\frac12}f_M(x/\sqrt{t})$ is the self-similar profile of the following viscous Burgers equation
	\begin{equation}\label{eq:limiteq}
		\begin{cases}
			u_t=\left(\frac{u^2}{2}\right)_x+2u_{xx},&(t,x)\in(0,\infty)\times\R, \\
			u(0,x)=M \delta_0,&x\in\R.
		\end{cases}
	\end{equation}
	where $M$ is the mass of the initial data.
\end{theorem}

Let us recall, see for example \cite{Aguirre:1990}, that equation \eqref{eq:limiteq}  has a unique solution of the form $u(t,x)=t^{-1/2}f_M(x/t^{1/2})$ where the profile $f_M$ is smooth, decays as $|x|\rightarrow\infty$ and has mass $M$ and constant sign.

 Moreover, the above theorem provides information about the asymptotic behavior of the approximation $Z^{n\dt}$ since $\uud(n\dt)=Z^{n\dt}u_0$. In fact, for any fixed $\Delta t$ and any $1\leq p<\infty$ we have
\begin{equation*}
	(n\dt)^{\frac12(1-\frac1p)}\left\|Z^{n\dt}u_0-u_M(n\dt )\right\|_{L^p(\R)}\to\infty,\quad\mbox{as }n\to\infty.
\end{equation*}

\begin{remark}
Note that the constant in front of the second derivative of $u$ in \eqref{eq:limiteq} is the sum of the coefficient already appearing in \eqref{eq:abe} plus one half of the second moment of function $K$. For general kernel $K\in L^1(\rr,1+x^2)$ the splitting approximation of the solutions of the equation
\begin{equation*}
	\begin{cases}
		\displaystyle u_t = \left(\frac{u^2}{2}\right)_x + u_{xx}+\int _{\rr}K(x-y)\big(u(y)-u(x)-(y-x)u_x(x)\big)dy, \\
		u(0,x)=u_0(x)
	\end{cases}
\end{equation*}
behaves for large time as the solution of the system
\begin{equation*}
	\begin{cases}
		u_t=\left(\frac{u^2}{2}\right)_x+(1+\frac{M_2}2)u_{xx},\\
		w(0)=M\delta_0.
	\end{cases}
\end{equation*}
where $M_2$ is the second moment of kernel $K$.
\end{remark}

The rest of the paper is organized as follows. In Section \ref{sec2} we show the convergence of the approximation $Z^{n\dt }u_0$ obtained by means of the splitting method, as well as its accuracy order. Then, in Section \ref{sec3}, we obtain the decay estimates of the $L^p$-norms of $u^\Delta $ in \eqref{u.delta}. We use them together with a scaling argument to prove Theorem \ref{teo:asymptotic}.
We conclude with some numerical experiments that illustrate our results in Section \ref{sec4}.

%%%%%%%%%%%% SECTION 2 %%%%%%%%%%%%%%%%%%%%%%%%%

\section{Convergence of the splitting method}\label{sec2}
This section is devoted to the proof of Theorem \ref{teo:order}. We first obtain some preliminary results on the flows $X^t$ and $Y^t$. We also consider the flow $S^tu_0$ and estimate it in the Sobolev spaces $H^s(\rr)$. 
These allow us to obtain local error estimates for flow $Z^t$ and, then, to obtain the order of convergence of the splitting method.

%%% SECTION 2.1
\subsection{Stability of the flows $X^t$ and $Y^t$}

Let us denote by $D_t$ the kernel of the semigroup corresponding to \eqref{eq:augmented}. It follows that $X^tv_0=D_t *v_0$. We first prove that the flow $X^t$ is stable both in $L^p(\R)$ and $H^s(\R)$ for any $1 \le p\le \infty$ and $s \ge 0$. Then, we obtain local error estimates for $X^t$.

\begin{lemma}\label{lemma:stab_xt}
	The flux $X^t$ corresponding to \eqref{eq:augmented} has the following properties:
	\begin{enumerate}
		\item For any $p\geq 1$ and $v_0\in L^p(\R)$, it holds that
			\begin{equation*}
				\|X^t v_0\|_{L^p(\R)}\le\|v_0\|_{L^p(\R)},\quad t\geq 0.
			\end{equation*}
		\item For any $s\ge0$ and $v_0\in H^s(\R)$, it holds that
			\begin{equation*}
				\|X^t v_0\|_{H^s(\R)}\le\|v_0\|_{H^s(\R)},\quad t\geq0.
			\end{equation*}
		\item For any integer $n\geq 1$, any real number $p\geq 1$ and $v_0\in W^{n,p}(\R)$, it holds that
			\begin{equation*}
				\|X^t v_0\|_{W^{n,p}(\R)}\le\|v_0\|_{W^{n,p}(\R)},\quad t\geq0.
			\end{equation*}
	\end{enumerate}
\end{lemma}
\begin{proof}
Let us focus on the first property. For $p=1$ we multiply equation \eqref{eq:augmented} by $\sgn(v)$ and integrate it on the whole space. Since $K$ has mass one, we obtain that the $L^1(\R)$-norm of $v$ decreases in time:
\begin{equation*}
	\frac{d}{dt}\int _{\R}|v(t)|dt=\int _{\R} (K\ast v-v)\sgn(v) \leq 0.
\end{equation*}
In the case $p=\infty$, we multiply by $\sgn(v-\mu)^+$, where $z^+:=\max\{0,z\}$ and $\mu=\|v_0\|_{L^\infty(\R)}$, and apply the same strategy.
For rigorous justifications of the above argument we refer to \cite{Escobedo:1991}.

Regarding the second property, using the Fourier transform, we have that
\begin{equation*}
	\widehat K(\xi)=\frac{1}{1+i\xi}, \ \xi\in \rr
\end{equation*}
and
\begin{equation}\label{eq:dtfourier}
	\widehat{D_t}(\xi) = e^{t(\widehat{K}(\xi)-1+i\xi)} = \exp\Big(\frac{-t\xi^2}{1+i\xi }\Big), \ \xi\in \rr.
\end{equation}
Thus, it is enough to observe that $|\widehat{D_t}(\xi)|\le1$.

For the last property it is sufficient to observe that $\partial_x$ commutes with $X^t$ and to use the first property.
\end{proof}

\begin{lemma}[Local error estimates for $X^t$]\label{lemma:dt-i}
	For any $v_0\in \dot H^1(\R)$ the following holds
	\begin{equation*}
		\|(X^t-I)v_0\|_{L^2(\R)}\le t \| v_0\|_{ \dot H^1(\R)}, \quad \forall t>0.
	\end{equation*}
\end{lemma}
\begin{proof}
From \eqref{eq:augmented}, we obtain that
%\begin{equation*}
%	(X^t-I)v_0=\int_0^tv_\tau(\tau)d\tau=\int_0^t\big(K*v(\tau)-v(\tau)+v_x(\tau)\big)d\tau
%\end{equation*}
%and, hence,
\begin{equation*}
	\|(X^t-I)v_0\|_{L^2(\R)}\le\int_0^t \|K*v(\tau)-v(\tau)+v_x(\tau)\|_{L^2(\R)} d\tau.
\end{equation*}
Using Plancherel's identity and \eqref{eq:dtfourier} we have for any $\tau\geq 0$
\begin{align*}
	\|K*v(\tau)-v(\tau)&+v_x(\tau)\|_{L^2(\R)}^2 = \int_\R\left|\frac{-\xi^2}{1+i\xi }\right|^2|\widehat{v}(\tau,\xi)|^2d\xi = \int_\R \frac{\xi^4}{1+\xi^2} |\widehat{v}(\tau,\xi)|^2d\xi \\
	& \le \int_\R \xi^2 \Big|\exp(\frac{-\tau \xi^2}{1+i\xi })\Big|^2|\widehat{v_0}(\xi)|^2d\xi \leq \|v_0\|_{\dot H^1(\R)}^2.
\end{align*}
Therefore, we deduce the desired estimate.
\end{proof}

Let us now focus on the flow $Y^t$ corresponding to the well-known viscous Burgers equation. For the sake of completeness, we recall the following estimates.

\begin{lemma}\label{lemma:stab_yt}
The flux $Y^t$ corresponding to \eqref{eq:burgers} has the following properties
\begin{enumerate}
	\item For any $p\geq 1$ and $w_0\in L^p(\R)$, it holds that
		\begin{equation*}
			\|Y^t w_0\|_{L^p(\R)}\le\|w_0\|_{L^p(\R)}, \quad t>0.
		\end{equation*}
	\item Let $n\geq 1$ an integer and $w_0\in H^n(\R)$. If $\|w_0\|_{L^2(\R)} \leq M$, then for any $T>0$ there exists $C=C(T,M)$ such that
		\begin{equation*}
			\|Y^t w_0\|_{H^n(\R)}\le e^{Ct}\|w_0\|_{H^n(\R)}, \quad \forall t\in[0,T].
		\end{equation*}
\end{enumerate}
\end{lemma}
\begin{proof}
The first property follows from classical estimates for convection-diffusion equations (e.g. \cite{Escobedo:1991}). The second one is, precisely, Corollary 3.8 in \cite{Bouharguane:2013}. 
\end{proof}

Using the properties of flows $X^t$ and $Y^t$ in Lemma \ref{lemma:stab_xt} and Lemma \ref{lemma:stab_yt} we obtain similar properties for $Z^t$.
\begin{lemma}
	\label{prop.z}
	For any $r\in \{1,2\}$ the flow $Z^t$ satisfies the following estimate 
	\begin{equation*}
	\|Z^t u_0\|_{H^r(\R)}\leq C(\|u_0\|_{L^2(\R)}) \|u_0\|_{H^r(\R)}, \quad \forall t\in [0,1].
	\end{equation*}
	Moreover, 
	\begin{equation*}
		\|Z^t u_0\|_{L^2(\R)}\leq \|u_0\|_{L^2(\rr)}, \quad \forall t\geq 0.
	\end{equation*}
\end{lemma}

%%% SECTION 2.2
\subsection{The nonlinear flow $S^t$}
Let us now focus on the nonlinear flow $S^t$ given by \eqref{eq:abe}. We first obtain estimates on the $H^n(\R)$-norm, $n\geq 0$ being an integer. We also  study the $L^2(\rr)$ Lipschitz property of the flow $S^t$ in balls of $H^1(\R)$.

\begin{lemma}\label{lemma:stab_st}
	Let  $T>0$. For every $u_0\in H^n(\R)$, $n\geq 1$ an integer,   solution $S^tu_0$ of \eqref{eq:abe} satisfies
	\begin{equation*}
		\|S^t u_0\|_{H^n(\R)}\le C(T,\|u_0\|_{H^{n-1}(\R)}) \|u_0\|_{H^n(\R)}, \quad\forall t\in[0,T].
	\end{equation*}
	Moreover, 
	\begin{equation*}
		\|S^tu_0\|_{L^2(\rr)}\leq \|u_0\|_{L^2(\rr)}\quad\forall t \in \rr.
	\end{equation*}
\end{lemma}

\begin{proof}
The $L^2(\rr)$-stability of the flow $S^t$ has already been proved in \cite{Pozo:2014}. Indeed, multiplying equation \eqref{eq:abe} by $u(t)$ and using integration by parts one can obtain that
$$\|S_tu_0\|_{L^2(\rr)}\leq \|u_0\|_{L^2(\rr)}.$$

In the following we will denote by $G_t$ the heat kernel. We now recall some well-known results about its $L^p(\rr)$-norms:
\begin{equation}
\label{decay.heat}
  \|\partial_x^kG_t\|_{L^p(\rr)}\simeq t^{-\frac12(1-\frac 1p)-\frac k2}, \quad, \forall\, k\geq 0,k\in \mathbb{Z}, \, \forall \,p\in [1,\infty]. 
\end{equation}

Let us now consider the case $n\geq 1$. The variation of constants formula gives us that $S^tu_0=u(t)$ satisfies
\begin{equation*}
	u(t)=G_t*X^t u_0+\frac 12\int _0^t G_{t-\tau} * X^{t-\tau} (u^2)_x(\tau)d\tau.
\end{equation*}
Differentiating in the space variable $x$ and using that $X^t$ commutes with $\partial_x$, we obtain 
\begin{equation*}
	u_x(t)=G_t*X^t(\partial_x u_0)+ \frac12 \int_0^t \partial_x G_{t-\tau}*X^{t-\tau}(uu_x(\tau))d\tau.
\end{equation*}
Using Young's inequality, Lemma \ref{lemma:stab_xt} and the  estimates in \eqref{decay.heat} we have that
\begin{align*}
	\|u_x(t)\|_{H^{n-1}(\R)} & \lesssim \|G_t\|_{L^1(\R)}\|X^t(\partial_x u_0)\|_{H^{n-1}(\R)} \\
	&\qquad \qquad \qquad + \int_0^t \|\partial_x G_{t-\tau}\|_{L^2(\R)} \|X^{t-\tau}(uu_x(\tau))\|_{W^{n-1,1}(\R)} d\tau \\
	& \lesssim \|\partial_x u_0\|_{H^{n-1}(\R)} + \int_0^t (t-\tau)^{-\frac34} \|X^{t-\tau}(uu_x(\tau))\|_{W^{n-1,1}(\R)} d\tau\\
		& \lesssim \|u_0\|_{H^{n}(\R)} + \int_0^t (t-\tau)^{-\frac34} \| uu_x(\tau)\|_{W^{n-1,1}(\R)} d\tau.
\end{align*}
When $n=1$, the H\"older inequality and the $L^2(\R)$-stability of the flow $S^t$  yield
\begin{align*}
	\|u_x(t)\|_{L^2(\R)} & \le \|u_0\|_{H^1(\R)} + C \int_0^t (t-\tau)^{-\frac34} \|u(\tau)\|_{L^2(\R)} \|u_x(\tau)\|_{L^2(\R)} d\tau \\
	& \le \|u_0\|_{H^1(\R)} + C \|u_0\|_{L^2(\R)} \int_0^t (t-\tau)^{-\frac34}\|u_x(\tau)\|_{L^2(\R)} d\tau.
\end{align*}
Fractional Gronwall's Lemma \cite[Lemma 2.4]{Bouharguane:2013} gives us the desired estimate when $n=1$. The case $n>1$ follows by induction.
\end{proof}

\begin{lemma}\label{lemma:slip}
	Let $R,T>0$. There exist $\gamma=\gamma(R,T)<\infty$ such that if
	\begin{equation*}
		\|u_0\|_{H^1(\R)}\le R \quad\mbox{ and }\quad \|v_0\|_{H^1(\R)}\le R,
	\end{equation*}
	then
	\begin{equation*}
		\|S^tu_0-S^tv_0\|_{L^2(\R)}\le \gamma\|u_0-v_0\|_{L^2(\R)},\quad\forall t\in[0,T].
	\end{equation*}
\end{lemma}
\begin{proof}
Let $u$ and $v$ be the corresponding solutions of \eqref{eq:abe} with $u_0$ and $v_0$ initial data. Writing the variation of parameters formula we obtain that
\begin{equation*}
	(u-v)(t)=G_t\ast X^t (u_0-v_0)+\frac 12\int _0^t G_{t-s}\ast X^{t-s} (u^2-v^2)_x(s)ds.
\end{equation*}
Using the decay of $G_t$ and its first derivative   in the $L^1(\rr)$-norm given by \eqref{decay.heat} we obtain 
\begin{align*}
	\|(u-v)(t)\|_{L^2(\rr)}&\leq \|G_t\ast X^t (u_0-v_0)\|_{L^2(\rr)}+\frac 12\int _0^t \|(G_{t-s})_x\ast X^{t-s} (u^2-v^2)(s)\|_{L^2(\rr)}ds\\
	& \lesssim \|u_0-v_0\|_{L^2(\rr)}+\int _0^t (t-s)^{-\frac12}\|(u^2-v^2)(s)\|_{L^2(\rr)}\\
	&\lesssim \|u_0-v_0\|_{L^2(\rr)}+\int _0^t (t-s)^{-\frac12}\|(u-v)(s)\|_{L^2(\rr)} (\|u(s)\|_{H^1(\rr)}+\|v(s)\|_{H^1(\rr)})ds.
\end{align*}
By Lemma \ref{lemma:stab_st} there exists a constant $C=C(T,\|u_0\|_{L^2(\rr)}, \|v_0\|_{L^2(\rr)})$ such that
\begin{align*}
	\|(u-v)(t)\|_{L^2(\rr)}&\leq \|u_0-v_0\|_{L^2(\rr)}+C (\|u_0\|_{H^1(\rr)}+\|v_0\|_{H^1(\rr)}) \int _0^t (t-s)^{-\frac12}\|(u-v)(s)\|_{L^2(\rr)} ds.
\end{align*}
Fractional Gronwall's lemma yields the result.
\end{proof}

%%% SECTION 2.3
\subsection{Local error estimates}
Before proving Theorem \ref{teo:order}, we obtain the error estimates  on $Z^t$ in  the $L^2(\rr)$-norm. We state this result in the following proposition.
\begin{prop}\label{prop:l2locz}
	Let $r\in \{1,2\}$ and $u_0\in H^r(\R)$. Then, there exists a positive constant $C=C(\|u_0\|_{L^2(\R)})$ such that for every $t\in(0,1)$  the following holds
	\begin{equation}\label{eq:l2locz}
		\|S^t u_0-Z^tu_0\|_{L^2(\R)}\le C \|u_0\|_{H^r(\R)}^2 t^{1+\frac r2}.
	\end{equation}
\end{prop}	
	
\begin{proof}
We write the variation of constants formula both for $S^tu_0=u(t)$ solution of \eqref{eq:abe} and $Z^t$ given by \eqref{eq:zt}. We have that they satisfy
\begin{equation*}
	u(t)=D_t*G_t*u_0+\frac12 \int_0^t D_{t-s}*G_{t-s}*\left(u^2(s)\right)_x ds,
\end{equation*}
and
\begin{equation*}
	Z^tu_0=X^tY^tu_0=D_t*Y^tu_0=D_t*\left(G_t*u_0+\frac12 \int_0^t G_{t-s}*\left((Y^su_0)^2\right)_x ds\right).
\end{equation*}
Therefore,
\begin{align}\label{eq:z-u}
	\|u(t)-Z^tu_0\|_{L^2(\R)} &\le \frac12 \int_0^t \left\|D_{t-s}*G_{t-s}*\left(u^2(s)\right)_x-D_t*G_{t-s}*\big((Y^su_0)^2\big)_x\right\|_{L^2(\R)}ds \\
	&\le \frac12\int_0^t \Big( \| R_1(s) \|_{L^2(\R)} + \| R_2(s) \|_{L^2(\R)} + \| R_3(s) \|_{L^2(\R)} \Big)ds, \notag
\end{align}
where
\begin{align*}
	R_1(s) &= D_{t-s}*G_{t-s}*\left(u^2(s)\right)_x - D_{t-s}*G_{t-s}*\left((Z^su_0)^2\right)_x, \\
	R_2(s) &= D_{t-s}*G_{t-s}*\left((Z^su_0)^2\right)_x - D_t*G_{t-s}*\left((Z^su_0)^2\right)_x, \\
	R_3(s) &= D_t*G_{t-s}*\left((Z^su_0)^2\right)_x-D_t*G_{t-s}*\left((Y^su_0)^2\right)_x.
\end{align*}

We now estimate each of these terms. 

\noindent\textit{Estimates on $R_1$}. Using Young's inequality and the $L^2(\rr)$-estimates on $\partial_xG_t$, $u(t)$ and $Z^su_0$ we obtain
\begin{align}\label{est.r0}
	\|R_1(s)\|_{L^2(\R)} &\le \left\|D_{t-s}\|_{L^1(\rr)}\|\partial_x G_{t-s}\right\|_{L^2(\R)}\left\|u^2(s)-(Z^su_0)^2 \right\|_{L^1(\R)} \\
	&\lesssim (t-s)^{-\frac34}\left\|u(s)-Z^su_0 \right\|_{L^2(\R)} (\|u(s)\|_{L^2(\R)}+\left\|Z^su_0\right\|_{L^2(\R)} ) \notag\\
	&\lesssim (t-s)^{-\frac34} \|u_0 \|_{L^2(\R)} \left\|u(s)-Z^su_0 \right\|_{L^2(\R)} .\notag
\end{align}

\noindent\textit{Estimates on $R_2$}.  For any $t\in(0,1)$, from Lemma \ref{lemma:dt-i} we have
\begin{align*}
	\|R_2(s)\|_{L^2(\R)}&= \|(D_{t-s} - D_t)*G_{t-s}*\left((Z^su_0)^2\right)_x\|_{L^2(\R)}
	\lesssim s\|G_{t-s}* \left((Z^su_0)^2\right)_x \|_{\dot H^1(\R)}	\\
	&=s\|G_{t-s}* (Z^su_0)^2\|_{\dot H^2(\R)}	=s\|\partial_x^{2-r} G_{t-s}* \partial_x^{r}\left((Z^su_0)^2\right) \|_{L^2(\R)} \\
	&\leq s\|\partial_x^{2-r} G_{t-s}\|_{L^1(\R)}\| \partial_x^{r}((Z^su_0)^2) \|_{L^2(\R)}	 \\
	&\leq s(t-s)^{-\frac{2-r}2} \|(Z^su_0)^2\|_{H^r(\R)}\leq s(t-s)^{-\frac{2-r}2} \|Z^su_0\|_{H^r(\R)}^2.
\end{align*}
Integrating on the time interval $[0,t]$ and using the estimates on the $H^r(\rr)$-norm of $X^t$ and $Y^t$ in Lemma \ref{lemma:stab_xt} and Lemma \ref{lemma:stab_yt} we obtain
\begin{align}\label{est.r1}
	\int _0^t \|R_2(s)\|_{L^2(\R)} ds & \lesssim \|Z^s u_0\|_{L^\infty((0,t),H^r(\R))}^2 \int _0^t s(t-s)^{-\frac{2-r}2}ds \\
	&\leq C(\|u_0\|_{L^2(\R)}) \|u_0\|_{H^r(\R)}^2 t^{1+\frac r2}. \notag
\end{align}

\noindent\textit{Estimates on $R_3$}. Observe that
\begin{align*}
	\|R_3(s)\|_{L^2(\R)}&= \|D_t* G_{t-s}* \left((Z^su_0)^2-(Y^su_0)^2\right)_x\|_{L^2(\R)} \\
	& \leq \| G_{t-s}* \left((Z^su_0)^2-(Y^su_0)^2\right)_x\|_{L^2(\R)}.
\end{align*}
We  distinguish between the cases $r=1$ and $r=2$. When $r=1$, using Young's inequality, Lemma \ref{lemma:stab_yt} and the embedding $H^1(\rr) \hookrightarrow L^\infty(\rr)$ 
we obtain:
\begin{align*}
	\|R_3(s)\|_{L^2(\R)} &\leq \| \partial_x G_{t-s}\|_{L^1(\R)} \| (Z^s u_0)^2-(Y^s u_0)^2\|_{L^2(\R)} \\
	&\lesssim (t-s)^{-\frac12}\| X^sY^s u_0 - Y^s u_0\|_{L^2(\R)} (\|Z^s u_0\|_{L^\infty(\R)} +\|Y^s u_0\|_{L^\infty(\R)} )\\
	&\lesssim s (t-s)^{-\frac12} \|Y^s u_0\|_{H^1(\R)}^2.
 \end{align*}
%On the other hand, when $r=2$ we cannot use the same argument as before since the estimate
%\begin{equation*}
%	\| Z^su_0 - Y^su_0\|_{L^2(\R)} \leq s \|Y^su_0\|_{H^2(\R)}
%\end{equation*}
%provides a similar estimate as in the case of $H^1(\R)$ initial data, without any improvement of the power of $s$, respectively $t$. 
For $r=2$ we proceed as follows:
\begin{align*}
\|R_3(s)\|_{L^2(\R)}&\leq \|G_{t-s}\|_{L^1(\R)} \| (Z^su_0)^2-(Y^su_0)^2)\|_{H^1(\R)}\\
	&\leq \|Z^su_0 - Y^su_0\|_{H^1(\R)} \|X^sY^su_0 + Y^su_0\|_{H^1(\R)} \\
	&\lesssim \|Z^su_0 - Y^su_0\|_{H^1(\R)} \| Y^su_0\|_{H^1(\R)} \lesssim s \| Y^su_0\|_{H^2(\R)}^2.
\end{align*}
Therefore, in both cases $r\in \{1,2\}$ we have
\begin{equation}\label{est.r2.2}
	\int _0^t \|R_3(s)\|_{L^2(\R)}\lesssim \|Y^su_0\|_{L^\infty((0,t),H^{r}(\R))}^2 \int _0^t (t-s)^{\frac {r-2}2}sds \le C(\|u_0 \|_{L^2(\R)})\|u_0\|_{H^{r}(\R)}^2 t^{1 +\frac r2}.
\end{equation}

Finally, plugging estimates \eqref{est.r0}, \eqref{est.r1} and \eqref{est.r2.2} into \eqref{eq:z-u}, we obtain
\begin{align*}
	\|u(t)-Z^tu_0\|_{L^2(\R)} &\le C(\|u_0 \|_{L^2(\R)})\|u_0\|_{H^{r}(\R))}^2 t^{1+\frac r2} \\
	&\qquad\qquad+ C(\|u_0 \|_{L^2(\R)}) \int_0^t (t-s)^{-\frac34}\left\|u(s) - Z^su_0 \right\|_{L^2(\R)} ds .
\end{align*}
Applying fractional Gronwall's Lemma \cite[Lemma 2.4]{Bouharguane:2013} we conclude that
\begin{align*}
	\|u(t)-Z^tu_0\|_{L^2(\R)} \le C(\|u_0 \|_{L^2(\R)})\|u_0\|_{H^{r}(\R)}^2 t^{1 +\frac r2}.
\end{align*}
The proof is now finished.
\end{proof}

%%% SECTION 2.4
\subsection{Proof of Theorem \ref{teo:order}}
We are now able to prove Theorem \ref{teo:order}. It is enough to follow ``Lady Windermere's fan'' argument (as in \cite{Holden:2013}; see also \cite{Bouharguane:2013}). Let us denote by $u^n=Z^{n\dt} u_0$ the approximate solution and $u^n_m=S^{(n-m)\dt} u^m$. Note that $u^n=u^n_n$ and $u(n\dt)=u^n_0$.

Let us first observe a few properties of the sequence $u^n =Z^{n\dt} u_0 $. Lemma \ref{prop.z} shows that
\begin{equation}\label{est.z.2}
	\|u^n\|_{L^2(\rr)}\leq \|u_0\|_{L^2(\rr)}, \quad \forall \, n\geq 0.
\end{equation}
Let us now consider the case of the $H^r(\rr)$ norms, $r\in \{1,2\}$. First, recall that by Lemma \ref{lemma:stab_yt} and estimate \eqref{est.z.2} there exists a positive constant $C=C(\|u_0\|_{L^2(\rr)})$ such that for any $\dt\leq 1$ the following holds:
\begin{equation*}
	\|Y^{\dt} u^m\|_{H^r(\rr)}\leq e^{C\dt} \|u^m\|_{H^r(\rr)}, \quad \forall \, m\geq 0.
\end{equation*}
This implies that
\begin{equation*}
	\|Z^{n\Delta t}u_0\|_{H^r(\rr)}=\|X^\dt Y^\dt u^{n-1}\| _{H^r(\rr)}\leq e^{C\dt} \|u^{n-1}\|_{H^r(\rr)}\leq e^{Cn\dt }\|u_0\|_{H^r(\rr)}, \quad \forall \, n\geq 0.
\end{equation*}
As long as $n\dt\leq T$ we get that $u^n=Z^{n\dt}u_0$ satisfies
\begin{equation}\label{est.z.hr}
	\|Z^{n\Delta t}u_0\|_{H^r(\rr)}\leq e^{CT}\|u_0\|_{H^r(\rr)}
\end{equation}
and the second part of Theorem \ref{teo:order} is proved.

We now proceed to obtain the global error estimate. We have
\begin{align}\label{eq:estimteo51}
	\|Z^{n\Delta t}u_0-u(n\Delta t)\|_{L^2(\R)} &\le \sum_{m=0}^{n-1} \|u^n_{m+1}-u^n_m\|_{L^2(\R)} \\
	& = \sum_{m=0}^{n-1} \|S^{(n-m-1)\dt} (Z^\dt u^m)-S^{(n-m-1)\dt} (S^{\dt} u^m)\|_{L^2(\R)}. \notag
\end{align}
In order to apply Lemma \ref{lemma:slip} we need to guarantee that $Z^\dt u^m=Z^{(m+1)\dt}u_0$ and $S^{\dt} u^m$ belong to some fix ball of $H^1(\rr)$ for $0\leq m\leq n-1$. For $Z^\dt u^m$ we  proved  such an estimate in  \eqref{est.z.hr}. For the second term, since $\dt\leq 1$, by Lemma \ref{lemma:stab_st} we have
\begin{equation*}
	\|S^\dt u^m\|_{H^1(\rr)}\leq C(1,\|u_0\|_{L^2(\rr)}) \|u^m\|_{H^1(\rr)} \leq C(\|u_0\|_{L^2(\rr)}) e^{TC(\|u_0\|_{L^2(\rr)})}\|u_0\|_{H^1(\rr)}.
\end{equation*}
 We now apply Lemma \ref{lemma:slip} with $r=1$. This  gives us that for some constant $C=C(T,\|u_0\|_{H^1(\rr)})$
\begin{equation*}
	\|Z^{n\Delta t}u_0-u(n\dt)\|_{L^2(\R)} \le C \sum_{m=0}^{n-1} \|Z^\dt u^m -S^{\dt} u^m \|_{L^2(\R)}.
\end{equation*}

Using the local error estimate obtained in Proposition \ref{prop:l2locz} we finally conclude that
\begin{equation*}
	\|Z^{n\Delta t}u_0-u(n\dt)\|_{L^2(\R)} \le C \sum_{m=0}^{n-1} \|u^m \|_{H^r(\R)}^2 (\dt)^{1+\frac{r}2}\leq C(T,\|u_0\|_{H^1}) \|u_0\|_{H^r(\rr)}^2 (\dt)^{\frac{r}2}.
\end{equation*}
The proof of the global error estimate is complete.

%%%%%%%%%%%% SECTION 3 %%%%%%%%%%%%%%%%%%%%%%%%%

\section{Large-time behavior}\label{sec3}
In this section, we first  analyze the decay properties of  $\uud$ and then we prove Theorem \ref{teo:asymptotic}. We first observe that $\uud$, defined by \eqref{u.delta}, 
satisfies the following equation
\begin{equation}\label{eq:spliteq}
	\begin{cases}
		\displaystyle \uud_t=2\psidt(t)\uud_{xx}+\psidt(t)((\uud)^2)_x+2\left(1-\psidt(t)\right)(K*\uud-\uud+\uud_x),&t>0, \\[10pt]
		\uud(0)=u_0,
	\end{cases}
\end{equation}
where  $\psidt(t)=\sum_{n\ge0} \chi_{I^\Delta_n}(t)$ and $I^{\Delta}_n=(t_n,t_{n+1/2})$. 

%%% SECTION 3.1
\subsection{Estimates for the $L^p(\R)$-norm.}
First, we need to obtain decay estimates for the $L^p$-norms of $\uud$. We proceed as in \cite{Ignat:2007,Ignat:2009}.
\begin{lemma}\label{lemma:decay}
For any $u_0\in L^1(\R)\cap L^\infty(\R)$ and $p\in [1,\infty)$ there exists a positive constant $C=C(p,\|u_0\|_{L^1(\R)},\|u_0\|_{L^\infty(\R)})>0$ such that the solution of system \eqref{eq:spliteq} satisfies
	\begin{equation}\label{decay.uud}
		\|\uud(t)\|_{L^p(\R)}\le C (t+1)^{-\frac12(1-\frac 1p)},\quad\forall t>0.
	\end{equation}
	Moreover, the $L^1(\rr)$-norm of $\uud$ does not increase.
\end{lemma}

\begin{proof}For $p=1$ the desired estimate on $\uud$ follows from the definition of  \eqref{u.delta} and the $L^1(\rr)$ stability property of the flows $X^t$ and $Y^t$. Let us now consider $p=2^k$ where $k\geq 1$ is an integer. We emphasize that in order to prove the decay property \eqref{decay.uud} it is sufficient to consider these particular cases of $p$. To simplify the presentation we will denote by $u$ the solution of equation \eqref{eq:spliteq}. Multiplying equation \eqref{eq:spliteq} by $u^{p-1}$ and integrating on $\R$, we obtain 
\begin{align}\label{eq:estimdisc1}
	\frac1p\frac{d}{dt}\|u(t)\|_{L^p(\R)}^p&=-C(p)\psidt(t)\|(u^{{p/2}})_x(t)\|_{L^2(\R)}^2+2(1-\psidt(t))\int_\R(K*u-u)u^{{p-1}}(t,x) dx\\
	\nonumber &=I_{1}+I_{2}.
\end{align}
Let us denote $J(z)=\frac12(K(z)+K(-z))$. We have that $J$ is an even function with mass one. Using the inequality (see \cite[Lemma 2.3]{Ignat:2015} for full details)
\begin{equation}\label{est.J.grad}
	\int_\R\int_\R J(x-y)\big(\varphi(x)-\varphi(y)\big)^2dxdy \le \int _{\R} J(z)z^2dz \int_\R\varphi_x^2dx = 2 \|\varphi_x\|_{L^2(\R)}^2,
\end{equation}
with $\varphi=u^{p/2}$ we obtain that the first term in the right hand side of \eqref{eq:estimdisc1} satisfies
\begin{equation}\label{ineg.501}
	I_{1}\leq - C(p,J) \psi^{\Delta t}(t) \int_\R\int_\R J(x-y)(u^{p/2}(x)-u^{p/2}(y))^2dxdy.
\end{equation}
We now prove a similar estimate for $I_{2}$. Elementary inequalities and the fact that $K$ has mass one show that the following holds:
\begin{equation*}
	\iint _{\rr^{2}} K(x-y)u(y)u^{p-1}(x)dxdy\leq \frac 2p \iint _{\rr^{2}} K(x-y)u^{{p/2}}(y)u^{p/2}(x)dxdy+ \left(1-\frac 2p\right) \int _{\rr}u^{p}(x)dx.
\end{equation*}
This implies that $I_{2}$ satisfies
\begin{equation*}
	I_{2}\leq \frac 4p \big(1-\psi^{\Delta t}(t)\big) \int _{{\rr}} (K\ast  u^{p/2}-u^{p/2})u^{p/2}.
\end{equation*}
Note that for any function $\varphi\in L^2(\rr)$ we have
\begin{align}\label{simetrizare}
	\int_\R(K*\varphi-\varphi)\varphi dx& =-\frac12\int_\R\int_\R K(x-y)\big(\varphi(x)-\varphi(y)\big)^2dxdy \nonumber\\
	&=-\frac12\int_\R\int_\R J(x-y)\big(\varphi(x)-\varphi(y)\big)^2dxdy
\end{align}
Therefore, we have that $I_{2}$ satisfies 
\begin{equation}\label{ineg.502}
I_{2}\leq -\frac 2p (1-\psi^{\Delta t}(t))  \int_\R\int_\R J(x-y)\big(u^{p/2}(x)-u^{p/2}(y)\big)^2dxdy.
\end{equation}

Using estimates \eqref{ineg.501} and \eqref{ineg.502} in \eqref{eq:estimdisc1} we obtain
%\begin{align*}
%	\frac12\frac{d}{dt}&\|\uud(t)\|_{L^2(\R)}^2= \\
%	&-2\psidt(t)\|\uud_x(t)\|_{L^2(\R)}^2-(1-\psidt(t))\int_\R\int_\R J(x-y)(\uud(t,x)-\uud(t,y))^2dxdy.
%\end{align*}
%Using an integral Taylor expansion for function $\varphi$ we obtain that
%\begin{equation}\label{est.J.grad}
%	\int_\R\int_\R J(x-y)(\varphi(x)-\varphi(y))^2dxdy \le \int _{\R} J(z)z^2dz \int_\R\varphi_x^2dx = 2 \|\varphi_x\|_{L^2(\R)}^2.
%\end{equation}
that for some positive constant $C(J,p)$ we have
\begin{equation}\label{ineg.503}
	\frac{d}{dt}\|u(t)\|_{L^p(\R)}^p\leq -C(J,p)\int_\R\int_\R J(x-y)\big(u^{{p/2}}(t,x)-u^{p/2}(t,y)\big)^2dxdy.
\end{equation}
When $p=2$ following the same arguments as in \cite{Ignat:2007,Ignat:2009} we conclude that
\begin{equation*}
	\|\uud(t)\|_{L^2(\R)} \lesssim \frac{\|u_0\|_{L^2(\R)}}{(t+1)^{1/2}}+\frac{\|u_0\|_{L^1(\R)}}{(t+1)^{1/4}}.
\end{equation*}
Using estimate \eqref{ineg.503} and  proceeding as in \cite{Ignat:2007}, we obtain  the estimates for the remaining $L^p(\R)$-norms for $p>2$.
\end{proof}

%%% SECTION 3.2
\subsection{Estimates on the rescaled solutions.}

In order to prove Theorem \ref{teo:asymptotic} we introduce the following family of rescaled functions. Given $\lambda>0$, we define $\ul$ by
\begin{equation}\label{eq:ul}
	\ul(t,x)=\lambda \uud(\lambda^2 t,\lambda x).
\end{equation}
It follows that $\ul$ is the solution of the following rescaled system
\begin{equation}\label{eq:ulambda}
	\begin{cases}
		\ul_t=2\psil(t)\ul_{xx}+2\psil(t)\ul\ul_x+2\left(1-\psil(t)\right)(\lambda^2 K_\lambda*\ul-\lambda^2\ul+\lambda \ul_x),&t>0, \\[10pt]
		\ul(0,x)=\lambda u_0(\lambda x),
	\end{cases}
\end{equation}
where $K_\lambda(z)=\lambda K(\lambda z)$.  We emphasize that in order to obtain the equation satisfied by $u^\lambda$ we used that 
\begin{equation*}
  \psidt(\lambda^2 t)=\sum_{n\ge1} \chi_{(n\dt<\lambda^2 t<(n+\frac12)\dt)} = \sum _{n\ge1} \chi_{(\frac{\dt}{\lambda^2}<t<(n+\frac12)\frac{\dt}{\lambda^2})} = \psil(t).
\end{equation*}

The results in Lemma \ref{lemma:decay} extend immediately to the family $\{\ul\}_{\lambda>0}$.
\begin{lemma}\label{lemma:uldecay}
Let $u_0\in L^1(\R)\cap L^\infty(\R)$. For any $p\in [1,\infty)$ there exists a positive constant $C=C(p,\|u_0\|_{L^1(\R)},\|u_0\|_{L^\infty(\R)})$
such that 
	\begin{equation*}
		\|\ul(t)\|_{L^p(\R)} \le C t^{-\frac 12(1-\frac 1p)},
	\end{equation*}
holds for any $\lambda>0$. Moreover, the mass of $\ul$ is conserved.
\end{lemma}
\begin{proof}
By the definition of $\ul$ and Lemma \ref{lemma:decay}, we have
\begin{equation*}
	\|\ul(t)\|_{L^p(\R)} = \lambda ^{1-\frac 1p}\|\uud(\lambda^2 t)\|_{L^p(\R)} \le C \left(t+\lambda^{-2}\right)^{-\frac 12(1-\frac 1p)} \le C t^{-\frac 12(1-\frac 1p)}.
\end{equation*}
The proof is now finished. The mass conservation follows from the same property of $X^t$ and $Y^t$. 
\end{proof}

In order to prove the compactness of the family $\{\ul \}_{\lambda>0}$, we recall the following compactness argument in \cite{Ignat:2015}:
\begin{theorem}[{\cite[Theorem 1.1]{Ignat:2015}}]\label{rossi-time}
Let $1< p<\infty$ and $\Omega\subset \R^d$ be an open set. Let $\rho:\R^d\rightarrow \R$ be a nonnegative smooth radial function with compact support, non identically zero, and $\rho_n(x)=n^d\rho(nx)$. Let $\{f_n\}_{n\geq 1}$ be a sequence of functions in $L^p((0,T)\times \Omega)$ such that
\begin{equation*}
	\int _0^T \int _{\Omega} |f_n(t,x)|^pdxdt \leq \ {M}
\end{equation*}
and
\begin{equation*}
	n^p\int _0^T\int _{\Omega}\int _{\Omega} \rho_n(x-y) |f_n(t,x)-f_n(t,y)|^pdxdydt\leq {M}.
\end{equation*}
Then, we have:
\begin{enumerate}
	\item If $\{f_n\}_{n\geq 1}$ is weakly convergent in $L^p((0,T)\times \Omega)$ to $f$, then $f\in L^p((0,T),W^{1,p}(\Omega))$ 
		\item   Assuming that $\Omega$ is a smooth bounded domain in $\R^d$, $\rho(x)\geq \rho(y)$ if $|x|\leq |y|$ and that
		\begin{equation*}
			\|\partial_t f_n\|_{L^p((0,T),W^{-1,p}(\Omega))}\leq M,
		\end{equation*}
		then $\{f_n\}_{n\geq 1}$ is relatively compact in $L^p((0,T)\times \Omega)$.
\end{enumerate}
\end{theorem}

We will apply the above theorem with $p=2$ to obtain the local compactness in $L^2_{loc}((0,\infty)\times\R)$. The boundedness of $\{u_\lambda\}_{\lambda>0}$ in $L^2((t_1,t_2)\times\rr)$ follows from Lemma \ref{lemma:uldecay}. We now check the other assumptions in Theorem \ref{rossi-time}. In what follows we denote $J_\lambda (z)=\lambda J(\lambda z)$.

\begin{lemma}\label{lemma:jlambda}
There exists a positive constant $C=C(J,\|u_0\|_{L^1(\rr)}, \|u_0\|_{L^\infty(\rr)})$ such that 	for any $0<t_1<t_2<\infty$,  the following estimate
	\begin{equation*}
		\lambda^2 \int_{t_1}^{t_2} \int_\R \int_\R J_\lambda(x-y)(\ul(t,x)-\ul(t,y))^2dxdydt \le C t_1^{-\frac12}
	\end{equation*}
	 holds for all $\lambda>0$.
\end{lemma}
\begin{proof}
Let us multiply \eqref{eq:ulambda} by $\ul$ and integrate it on space. Using \eqref{simetrizare} we get
\begin{align*}
	\frac 14\frac{d}{dt} \int_\R |\ul |^2 dx&=-\psil(t)\int_\R |\ul_x |^2dx+\lambda^2(1-\psil(t))\int_\R (K_\lambda*\ul-\ul)\ul dx \\
	&=-\psil(t)\int_\R |\ul_x |^2dx-\frac{\lambda^2}{2}(1-\psil(t))\int_\R\int_\R J_\lambda(x-y) (\ul(x)-\ul(y))^2 dxdy.
\end{align*}
Integrating on $(t_1,t_2)$ the above inequality and using the decay of the $L^2(\rr)$-norm of $u^\lambda$ we obtain 
\begin{align}\label{est.energetica}
	\int_{t_1}^{t_2} \psil(t)\int_\R & |\ul_x |^2dxdt+\lambda^2\int_{t_1}^{t_2}(1-\psil(t))\int_\R\int_\R  J_\lambda(x-y) (\ul(x)-\ul(y))^2 dxdydt\\
	\nonumber&\lesssim \|\ul(t_1)\|_{L^2(\R)}^2-\|\ul(t_2)\|_{L^2(\R)}^2 \leq  C(\|u_0\|_{L^1(\rr)}, \|u_0\|_{L^\infty(\rr)}) t_1^{-\frac12}.
\end{align}
By \eqref{est.J.grad} and a change of variables, we also have
\begin{equation*}
	\lambda^2 \int_\R\int_\R J_\lambda(x-y) (\ul(x)-\ul(y))^2 dxdy \leq C(J) \int_\R |\ul_x (x)|^2 dx.
\end{equation*}
We introduce this inequality in \eqref{est.energetica} to obtain the desired result.
%
%This implies that
%\begin{equation*}
%	\lambda^2 \int_{t_1}^{t_2}\int_\R\int_\R J_\lambda(x-y) (\ul(t,x)-\ul(t, y))^2 dxdy \lesssim t_1^{-\frac12}.
%\end{equation*}
%and the proof of this estimate finishes.
\end{proof}

We now prove that family $\{u_\lambda\}_{\lambda>0}$ satisfies the last assumption in Theorem \ref{rossi-time}.
%It remains to prove the last estimate of Theorem \ref{rossi-time}, regarding the $H^{-1}(\R)$ norm, in order to apply it to show the compactness of the set $\{\ul\}_{\lambda>0}$.
\begin{lemma}\label{lemma:h-1}
	There exists a positive constant $C=C(J,\|u_0\|_{L^1(\rr)}, \|u_0\|_{L^\infty(\rr)})$ such that for any $0<t_1<t_2<\infty$ the 
	 following 
	\begin{equation*}
		\|\ul_t\|_{L^2((t_1,t_2),H^{-1}(\R))}\le C t_1^{-\frac12}
	\end{equation*}
	holds uniformly for all $\lambda>1$.
\end{lemma}
\begin{proof}
Let us consider $\phi\in C^\infty_c(\R)$. Then
\begin{align*}
	\la \ul_t,\phi \ra_{-1,1} &= \Big\la 2\psil \ul_{xx} + 2\psil\ul\ul_x + 2(1-\psil)\left(\lambda^2\left(K_\lambda*\ul-\ul\right)+\lambda\ul_x\right),\phi\Big\ra _{-1,1} \\
	&= -2\psil \int_\R \ul_x \phi_x dx -\psil \int_\R (\ul)^2\phi_x dx \\
	&\quad \qquad + 2(1-\psil)\int_\R \left(\lambda^2 (K_\lambda*\ul-\ul )+\lambda\ul_x\right) \phi dx.
\end{align*}
Hence 
\begin{align*}
	|\la \ul_t ,\phi\ra _{-1,1} | &\lesssim \psil (t) \|\ul_x(t)\|_{L^2(\R)} \|\phi\|_{H^1(\R)} +\psil (t) \|(\ul)^2 (t)\|_{L^2(\R)} \|\phi\|_{H^1(\R)}\\
	&\quad + (1-\psil (t))\Big| \int _{\R} \Big( \lambda^2(K_\lambda\ast \ul -\ul)+\lambda \ul_x \Big)\phi dx\Big|.
\end{align*}
Let us now estimate the last term in the right hand side of the above inequality. Let us first recall that the Fourier transforms of $K$ and $J$ are given by
\begin{equation*}
	\widehat K(\xi)=\frac1{1+i\xi} \quad\mbox{ and }\quad \widehat J(\xi)=\frac{1}{1+\xi^2}.
\end{equation*}
The rescaled functions $K_\lambda$ and $J_\lambda$ satisfy
\begin{equation*}
	\widehat K_\lambda(\xi)=\frac1{1+i\xi/\lambda}  \quad\mbox{ and }\quad \widehat J(\xi)=\frac{1}{1+(\xi/\lambda)^2}.
\end{equation*}
Thus
\begin{align*}
	\bigg| \int_\R \Big(\lambda^2&\big(K_\lambda*\ul-\ul\big)+\lambda\ul_x\Big) \phi dx\bigg| \leq \int_\R \Big |\left(\lambda^2 (\widehat K_\lambda (\xi)-1)+\lambda i\xi\right)\widehat{\ul}(\xi)\overline{\widehat\phi(\xi)}\,d\xi \Big|\\
	&\leq \int _{\R} \frac{\lambda \xi^2}{\sqrt{\lambda^2 +\xi^2}}|\widehat{\ul}(\xi)||\widehat\phi(\xi)|\,d\xi \leq \left(\int _{\R} \frac{\lambda^2 \xi^2}{\lambda^2+\xi^2}|\widehat{\ul}(\xi)|^2 d\xi\right)^{1/2} \|\phi \|_{H^1(\R)}\\
	&= \left(\int _{\R} \big(1-\widehat J_\lambda(\xi)\big) |\widehat{\ul}(\xi)|^2 d\xi\right)^{1/2}\|\phi \|_{H^1(\R)} \\
	&= \left( \frac{\lambda^2}2 \int _{\R} J_\lambda (x-y)\big(\ul(x)-\ul (y)\big)^2dxdy \right)^{1/2}\|\phi \|_{H^1(\R)}.
\end{align*}
Hence
\begin{align*}
	|\la \ul_t ,\phi\ra _{-1,1} | &\lesssim \psil (t) \|\phi\|_{H^1(\R)} ( \|\ul_x(t)\|_{L^2(\R)}+ \|\ul (t)\|_{L^4(\R)}^2 )\\
	&\quad + (1-\psil (t))\Big( \frac{\lambda^2}2 \int _{\R} J_\lambda (x-y)\big(\ul(x)-\ul (y)\big)^2dxdy \Big)^{1/2}\|\phi \|_{H^1(\R)}.
\end{align*}
Using the $L^4(\rr)$-decay of $u_\lambda$ in Lemma \ref{lemma:uldecay} we have
\begin{align*}
	\|\ul _t(t)\|_{H^{-1}(\R)}& \lesssim \psil (t) \big(\|\ul_x(t)\|_{L^2(\R)} +  t^{-3/4}\big)\\
	&\quad + (1-\psil (t))\Big( \frac{\lambda^2}2 \int _{\R} J_\lambda (x-y)\big(\ul(x)-\ul (y)\big)^2dxdy \Big)^{1/2}.
\end{align*}
Integrating on time the above inequality we get
\begin{align*}
	\int _{t_1}^{t_2} \| \ul _t(t)\|_{H^{-1}(\R) }^2dt& \lesssim {t_1^{-\frac12}}+\int _{t_1}^{t_2}\psil (t) \int _{\R} |\ul_x (t)|^2dt\\
	&\quad + \frac{\lambda^2}2 \int _{t_1}^{t_2}(1-\psil (t)) \int _{\R} J_\lambda (x-y)\big(\ul(x)-\ul (y)\big)^2dxdy.
\end{align*}
Using now estimate \eqref{est.energetica} we conclude that 
\begin{equation*}
	\int _{t_1}^{t_2} \| \ul _t(t)\|_{H^{-1}(\R) }^2dt \lesssim {t_1^{-\frac12}},
\end{equation*}
which finishes the proof.
\end{proof}

Let us now choose a smooth compactly supported function $\rho\leq J$ as required in Theorem \ref{rossi-time}. The previous estimates obtained in Lemma \ref{lemma:uldecay}, Lemma \ref{lemma:jlambda}, Lemma \ref{lemma:h-1} and Theorem \ref{rossi-time} show that the family $\{u_\lambda\}_{\lambda>0}$ is relatively compact in $L^2_{loc}((0,\infty)\times \R)$ and so in $L^1_{loc}((0,\infty)\times \R)$. In order to prove the compactness in $L^1_{loc}((0,\infty),L^1(\R))$ we need to uniformly control the tails. Before proving that, we need the following comparison result.
\begin{lemma}\label{compar}
Let $u_0\leq v_0$ be two functions in $L^1(\R)\cap L^\infty(\R)$. Then, the solutions $\ul$ and $v^\lambda$ of system \eqref{eq:ulambda} with initial data $u_0$ and $v_0$, respectively, satisfy
\begin{equation*}
	\ul(t,x)\leq v^\lambda(t,x).
\end{equation*}
\end{lemma}
\begin{proof}
Let us set $w^\lambda=\ul-v^\lambda$. From \eqref{eq:ulambda} we obtain the equation satisfied by $w^\lambda$ and multiply it by $\sgn^+(w^\lambda)$, where function $\sgn^+(z)=1$ if $z$ is positive and vanishes otherwise. Then the positive part of $w$ satisfies $(w^\lambda)^+(0)=0$ and
\begin{align*}
	\frac{d}{dt}\int_{\R} (w^\lambda)^+dx=& 2\psi^{\frac{\dt}{\lambda^2}} (t) \int _{\R}w^\lambda_{xx}\sgn^+(w^\lambda)dx+\psi^{\frac{\dt}{\lambda^2}}(t) \int _{\R} \big((\ul)^2-(v^\lambda)^2\big)_x\sgn^+(w^\lambda)dx\\
	&+2\big(1-\psi^{\frac{\dt}{\lambda^2}}(t)\big)\int _{\R}(\lambda^2K_\lambda\ast w^\lambda-\lambda^2w^\lambda+\lambda w^\lambda_x)\sgn^+(w^\lambda)dx.
\end{align*}
Classical density arguments (see, for example, \cite{Escobedo:1991}) show that
\begin{equation*}
	\int _{\R}w^\lambda_{xx}\sgn^+(w^\lambda)\leq 0,\quad \int _{\R} \big((\ul)^2-(v^\lambda)^2\big)_x \sgn^+(w^\lambda)dx=0, \quad \int _{\R} w^\lambda_x\sgn^+(w^\lambda)dx=0.
\end{equation*}
Let us now estimate the last term. We have
\begin{align*}
	\int _{\R} (K_\lambda\ast w^\lambda-w^\lambda) \sgn^+(w^\lambda)dx &=\int_{\R} \int _{\R} K_\lambda(x-y)w^\lambda(y)\sgn^+ (w^\lambda(x))dxdy-\int _{\R}(w^\lambda)^+ dx\\
	&\leq \int_{\R} \int_{\R} K_\lambda(x-y)(w^\lambda)^+(y)dxdy-\int _{\R}(w^\lambda)^+dx \\
	&=\int _{\R}(w^\lambda)^+(y)\int _{\R}K_\lambda(x-y)dxdy-\int _{\R}(w^\lambda)^+dx=0.
\end{align*}
This finishes the proof.
\end{proof}

To conclude this subsection, let us show that the tails are uniformly small when $\lambda\ge1$, a sufficient condition to prove, afterwards, that $\{\ul\}_{\lambda>0}$ converges in $L^1_{loc}((0,\infty),L^1(\R))$.

\begin{lemma}\label{lemma:tails}
For any  $u_0\in L^1(\R)\cap L^\infty(\R)$ there exists a positive constant $C=C(\|u_0\|_{L^1(\R)}, \|u_0\|_{L^\infty(\R)} )$ such that, for any $\lambda\ge1$ and any $r>0$, the following holds:
\begin{equation*}\label{eq:tails}
	\int _{|x|>2r} |\ul (t,x)|dx\leq \int _{|x|>r} |u_0(x)|dx+ C\left(\frac {\sqrt{t}}{r} +\frac{t}{r^2}\right),
\end{equation*}
where $\ul$ is the solution of system \eqref{eq:ulambda} with initial data $u_0^\lambda(\cdot)=\lambda u_0(\lambda \cdot)$.
\end{lemma}
\begin{proof}
In order to prove the above estimate we first remark that it is sufficient to consider positive solutions. The maximum principle given in Lemma \ref{compar} shows that $v\leq \ul\leq w$ where $v$ and $w$ are solutions of \eqref{eq:ulambda} with initial data $-|u_0^\lambda|$ and $|u_0^\lambda|$ respectively. For $w$ we need estimates for nonnegative solutions whereas for $v$ we have to observe that $-v$ is the solution of the equation with nonnegative initial data similar to \eqref{eq:ulambda} but replacing the term $uu_x$ by $-uu_x$. This change does not affect the estimates in Lemma \ref{lemma:tails}. Then $|u_\lambda|\leq \max \{|v|,w\}$. Taking into account this argument, it is enough to prove the desired estimate only for nonnegative solutions.

For every $r>0$, let us define $\rho_r(z)=\rho(z/r)$, where $\rho$ is a nonnegative $C^\infty$ function such that $0\leq \rho\leq 1$ and
\begin{equation*}
	\rho(z)=\begin{cases} 0,&|z|<1,\\1,&|z|>2. \end{cases}
\end{equation*}
Let us assume that $u_\lambda$ is a nonnegative solution. We multiply \eqref{eq:ulambda} by $\rho_r$ and integrate it over $(0,t)\times\R$. We obtain:
\begin{align*}
	\int_0^t\int_\R \ul_s (s) \rho_r dxds&=-\int_0^t\psil(s) \int_\R \big(\ul (s)\big)^2 \rho_r' dxds + 2\int_0^t \psil(s) \int_\R \ul (s) \rho_r'' dxds\\
	&\quad\quad+2\int_0^t \big(1-\psil(s)\big) \int_\R \Big(\lambda^2\big(K_{\lambda}*\ul(s)- \ul(s)\big)+\lambda \ul_x(s) \Big)\rho_r dxds
\end{align*}
and, therefore,
\begin{align}\label{eq:int_j}
	\int_\R \ul(t) \rho_r dx&\le \int_\R \ul_0 \rho_r dx+\frac{\|\rho'\|_{L^\infty(\R)}}{r}\int_0^t \|\ul(s)\|_{L^2(\R)}^2ds+\frac{\|\rho''\|_{L^\infty(\R)}}{r^2}\int_0^t\|\ul(s)\|_{L^1(\R)}ds \\
	&\quad+2\int_0^t \big(1-\psil(s)\big) \int_\R \Big(\lambda^2\big(K_{\lambda}*\ul(s)- \ul(s)\big)+ \lambda \ul_x(s)\Big)\rho_r dx ds. \notag
\end{align}
We have to obtain an estimate on the last term, uniformly on $\lambda$. The same argument as in \cite[Theorem 2.5]{Pozo:2014} shows that
\begin{equation*}
	\int_\R \Big(\lambda^2\big(K_{\lambda}*\ul(s)- \ul(s)\big)+ \lambda \ul_x(s)\Big)\rho_rdx \lesssim \|u_0\|_{L^1(\rr)}\|\rho_r''\|_{L^\infty(\rr)}
	\lesssim 
	\frac{\|u_0\|_{L^1(\rr)} \|\rho''\|_{L^\infty(\rr)}}{r^2}.
\end{equation*}
So, plugging this into \eqref{eq:int_j} and using Lemma \ref{lemma:decay}, we get:
\begin{equation*}
	\int_\R \ul(t,x)\rho_r (x)dx \le \int_\R  u_{0}(x)\rho_{\lambda r}(x) dx+C\left(\frac{\sqrt{t}}{r}+\frac{t}{r^2}\right)
\end{equation*}
where constant $C$ depends only on $\|u_0\|_{L^1(\rr)}, \|u_0\|_{L^\infty(\rr)}$ and $\|\rho\|_{W^{2,\infty}(\R)}$, which are both bounded. For $\lambda\ge 1$, since  $(-r,r)\subset {\rm{supp}} (\rho _{\lambda r})$, we finally conclude
\begin{equation*}
	\int_{|x|>2r} |\ul(t,x)| dx \le \int_{|x|>r} \left|u_0(x)\right| dx+C\left(\frac{\sqrt{t}}{r}+\frac{t}{r^2}\right).
\end{equation*}
The proof is now finished.
\end{proof}

%%% SECTION 3.3
\subsection{Compactness in $L^1_{loc}((0,\infty),L^1(\R))$.}
We now proceed as in \cite{Ignat:2015}. For any $0<t_1<t_2<\infty$, Lemma \ref{lemma:decay} and Lemma \ref{lemma:h-1} give us the existence of a function $\bar u\in C([t_1,t_2],H^{-\eps}(\R))$, with $\eps\in(0,1)$, such that, up to a subsequence that we do not relabel,
\begin{equation*}
	\ul\to\bar u \quad\mbox{ in }\quad C([t_1,t_2],H^{-\eps}(\R)).
\end{equation*}
Moreover, for every $t>0$ and $p\in[1,\infty)$, we have that
\begin{equation*}
	\ul(t) \rightharpoonup \bar u(t) \quad\mbox{ in }\quad L^p_{loc}(\R).
\end{equation*}
In fact the bound from Lemma \ref{lemma:decay} transfers to $\bar u$ and, hence, 
\begin{equation}\label{eq:ubardecay}
	\|\bar u(t)\|_{L^p(\R)} \le \liminf_{\lambda\to\infty} \|\ul(t)\|_{L^p(\R)} \le C t^{-\frac12(1-\frac1p)},\quad\forall t>0.
\end{equation}
Now, from Lemma \ref{lemma:decay}, Lemma \ref{lemma:h-1} and Lemma \ref{lemma:tails} we deduce that for any $0<t_1<t_2<\infty$ there exists a positive constant $C=C(t_1,\|u_0\|_{L^1(\R)},\|u_0\|_{L^\infty(\R)})$ such that
\begin{equation*}
	\|\ul(t)\|_{L^2((t_1,t_2)\times \R)} \le C.
\end{equation*}
\begin{equation*}
	\lambda^2 \int_{t_1}^{t_2} \int_\R \int_\R J_\lambda(x-y)(\ul(t,x)-\ul(t,y))^2dxdydt \le C,
\end{equation*}
and
\begin{equation*}
	\|\ul_t\|_{L^2((t_1,t_2),H^{-1}(\R))}\le C.
\end{equation*}
Thus, we can apply Theorem \ref{rossi-time} to the family $\{\ul\}_{\lambda>0}$ and to the time interval $(t_1,t_2)$. We obtain that there exists a function $\omega \in L^2((t_1,t_2),H^1(\R))$ such that, up to a subsequence, 
\begin{equation*}
	\ul\to \omega \quad\mbox{ in }\quad L^2((t_1,t_2),L^2_{loc}(\R)).
\end{equation*}
In fact, due to the previous analysis, we know that $\omega=\bar u$. Therefore, we conclude that $\bar u\in L^2((t_1,t_2),H^1(\R))\cap L^1_{loc}((t_1,t_2)\times \R)$ and that
\begin{equation*}
	\ul\to \bar u \quad\mbox{ in }\quad L^1_{loc}((0,\infty)\times \R).
\end{equation*}
Finally, the strong convergence in $L^1_{loc}((0,\infty),L^1(\R))$ is an immediate consequence of Lemma \ref{lemma:tails}; indeed a standard diagonal argument reduces the compactness in $L^1_{loc}((0,\infty),L^1(\R))$ to prove that
\begin{equation*}
	\int_{t_1}^{t_2}\|\ul(t)\|_{L^1(|x|>r)}dt\to 0\quad\mbox{ as }\quad r\to\infty,\quad\mbox{ uniformly in }\lambda\ge1,
\end{equation*}
a property that follows from Lemma \ref{eq:tails}.

%%% SECTION 3.4
\subsection{Proof of Theorem \ref{teo:asymptotic}}
We can now prove the last main result of this paper. We proceed in several steps, as in \cite{Ignat:2015}.

\noindent\textit{Step I. Passing to the limit. }
Let us multiply equation \eqref{eq:ulambda} by a test function $\varphi\in C_c^\infty(\R)$ and integrate it on $(\tau,t)\times\R$, where $0<\tau<t$. We have:
\begin{align*}
	\int_\R \ul(t,x)&\varphi(x)dx - \int_\R \ul(\tau,x)\varphi(x)dx \\
	&=2 \int_\tau^t \psil(s) \int_\R \ul_{xx}(s,x) \varphi(x) dxds + 2\int_\tau^t \psil(s) \int_\R \ul(s,x)\ul_x(s,x) \varphi(x) dxds \\
	&\quad +2 \int_\tau^t \left(1-\psil(t)\right) \int_\R \left(\lambda^2 K_\lambda*\ul(s,x) -\lambda^2\ul(s,x) +\lambda \ul_x(s,x)\right) \varphi(x) dxds \\
	&=2 \int_\tau^t \psil(s) \int_\R \ul(s,x) \varphi''(x) dxds - \int_\tau^t \psil(s) \int_\R \big(\ul(s,x)\big)^2 \varphi'(x) dxds \\
	&\quad +2 \int_\tau^t \left(1-\psil(t)\right) \int_\R \ul(s,x) \left(\lambda^2 \tilde K_\lambda*\varphi(x) -\lambda^2\varphi(x) -\lambda \varphi'
	(x)\right) dxds\\
	&=I_1+I_2+I_3,
\end{align*}
where $\tilde K_\lambda(z) = K_\lambda(-z).$ We know that $\ul(t) \rightharpoonup \bar u(t)$ in $L^p_{loc}(\R)$, $1\leq p <\infty$, for all $t>0$. Thus,
\begin{equation*}
	\int_\R \ul(t,x)\varphi(x)dx - \int_\R \ul(\tau,x)\varphi(x)dx \to \int_\R \bar u(t,x)\varphi(x)dx - \int_\R \bar u(\tau,x)\varphi(x)dx.
\end{equation*}

Let us recall that $\psi^\dt \rightharpoonup\frac12$  weakly-* in $L^\infty_{loc}((0,\infty)\times \rr)$ (see \cite{Crandall:1980}, for example). Using that $\ul \rightarrow \bar u$ in $L^1_{loc}((0,\infty),L^1(\rr))$ we obtain that 
\begin{equation*}
	I_1 \to \int_\tau^t \int_\R \bar u (s,x) \varphi''(x) dxds, \ \text{as}\  \lambda\rightarrow\infty.
\end{equation*}

The strong convergence of $u^\lambda$ towards $\bar u$ in $L^2_{loc}((0,\infty)\times\rr)$ implies that $(u^\lambda)^2\rightarrow \bar u^2$ in $L^1_{loc}((0,\infty)\times\rr)$ and then 
\begin{equation*}
  I_2\rightarrow  \frac12 \int_\tau^t \int_\R \big(\bar u(s,x)\big)^2 \varphi'(x) dxds, \ \text{as}\  \lambda\rightarrow\infty.
\end{equation*}

Using the strong convergence of  $u^\lambda$ in $L^1_{loc}((0,\infty)\times \rr)$ and that, for any smooth function $\varphi$, we have the following estimate 
\begin{equation*}
	\|\lambda^2 \tilde K_\lambda*\varphi(x) -\lambda^2\varphi(x) -\lambda \varphi'(x) - \varphi'' \|_{L^\infty(\rr)}\lesssim \frac {\|\varphi ^{'''}\|_{L^\infty(\rr)}}{\lambda},
\end{equation*}
we obtain that
\begin{equation*}
	I_3\rightarrow \int_\tau^t \int_\R \bar u (s,x) \varphi''(x) dxds, \ \text{as}\  \lambda\rightarrow\infty.
\end{equation*}

In conclusion, all the above convergences show that $\bar u$ satisfies
\begin{align}
	\int_\R \bar u(t,x)&\varphi(x)dx - \int_\R \bar u(\tau,x)\varphi(x)dx \\
	&= 2 \int_\tau^t \int_\R \bar u (s,x) \varphi''(x) dxds - \frac12 \int_\tau^t \int_\R \big(\bar u(s,x)\big)^2 \varphi'(x) dxds. \notag
\end{align}
Note that this implies that $\bar u$ is a weak solution of $u_t=uu_x+2u_{xx}$.

\noindent\textit{Step II. Identification of the initial data. }
Let us choose $\tau=0$ in the previous step. Then, for any $\varphi\in C_b^2(\R)$, we have
\begin{align*}
	\bigg|\int_\R \ul(t,x)&\varphi(x)dx - \int_\R \ul(0,x)\varphi(x)dx\bigg| \\
	&\le 2 \bigg| \int_0^t \psil(s) \int_\R \ul(s,x) \varphi''(x) dxds\bigg| + \bigg|\int_0^t \psil(s) \int_\R \big(\ul(s,x)\big)^2 \varphi'(x) dxds\bigg| \\
	&\quad +2 \bigg|\int_0^t \left(1-\psil(t)\right) \int_\R \ul(s,x) \left(\lambda^2 \tilde K_\lambda*\varphi(x) -\lambda^2\varphi(x) -\lambda \varphi'(x)\right) dxds\bigg| \\
	&\lesssim \|\varphi\|_{W^{2,\infty}(\R)} \int_0^t \|\ul(s)\|_{L^1(\R)} ds + \|\varphi\|_{W^{1,\infty}(\R)} \int_0^t \| \ul(s)\|_{L^2(\R)}^2 ds ,
\end{align*}
where we used that (see \cite{Ignat:2015} for similar estimates) there exists a positive constant $C$ such that, for any $\lambda>0$ and $\varphi \in W^{2,\infty}(\rr)$, the following holds
\begin{equation*}
	\|\lambda^2 (K_\lambda \ast \varphi -\varphi )- \lambda \varphi ' \|_{L^\infty(\rr)}\leq C \|\varphi''\|_{L^\infty(\rr)}.
\end{equation*}
Using Lemma \ref{lemma:uldecay}, we deduce that
\begin{align*}
	\bigg|\int_\R \ul(t,x)\varphi(x)dx - \int_\R \ul(0,x)\varphi(x)dx\bigg| \lesssim \|\varphi\|_{W^{2,\infty}(\R)} (t+ t^{\frac12}).
\end{align*}
Letting $\lambda\to\infty$, we obtain that, for any $\varphi\in C_b^2(\R)$ 
\begin{equation*}
	\bigg|\int_\R \bar u(t,x)\varphi(x)dx - M\varphi(0) \bigg| \lesssim \|\varphi\|_{W^{2,\infty}(\R)} (t+ t^{\frac12}),
\end{equation*}
where $M=\int_\R u_0dx$ is the mass of the initial data $u_0$. This implies that, for any $\varphi\in C_b^2(\rr)$, we have
\begin{equation*}
	\lim_{t\rightarrow 0} \int_\R \bar u(t,x)\varphi(x)dx =M\varphi(0).
\end{equation*}

Standard density arguments and the tail control in Lemma \ref{lemma:tails}
show that the above limit holds for any bounded continuous function $\varphi$. We can conclude that $\bar u(t)$ goes to $M\delta_0$ as $t\to0$ in the sense of bounded measures. Therefore, we have proved that $\bar u\in L^\infty_{loc}((0,\infty),L^1(\R))\cap L^2_{loc}((0,\infty),H^1(\R))$ is the unique solution $u_M$ of the following viscous Burgers equation:
\begin{equation*}
	\begin{cases}
		u_t=uu_x+2u_{xx},&(t,x)\in(0,\infty)\times\R, \\
		u(0,x)=M\delta_0,& x\in\R.
	\end{cases}
\end{equation*}
Since this equation has a unique solution, then the whole sequence $\ul$ converges to $\bar u$, not only a subsequence.

\noindent\textit{Step III. Asymptotic behavior. }
The strong convergence of $u^\lambda$ toward $u_M$ in $L^1_{loc}((0,\infty),L^1(\R))$ also guarantees that for a.e. $t>0$ we have
\begin{equation}\label{eq:convl1}
	\|\ul(t)- u_M(t)\|_{L^1(\R)}\to 0\quad\mbox{ as }\quad\lambda\to\infty.
\end{equation}
In particular, there exists some $t_0>0$ such that
\begin{equation*}
	\|\ul(t_0)-u_M(t_0)\|_{L^1(\R)}\to 0\quad\mbox{ as }\quad\lambda\to\infty.
\end{equation*}
If we set $t=\lambda^2 t_0$ and use the self-similar form of $u_M$, we obtain
\begin{equation*}
	\lim_{t\to\infty }\| u (t)-u_M(t) \|_{L^1(\R)} =0.
\end{equation*}
This is exactly the case $p=1$ of \eqref{eq:convrates}. The general case $p\in[1,\infty)$ follows immediately, since
\begin{align*}
	\|u(t)-u_M\|_{L^p(\R)} &\lesssim \|u(t)-u_M\|_{L^1(\R)}^{\frac{1}{2p-1}}\left(\|u(t)\|_{L^{2p}(\R)}+\|u_M(t)\|_{L^{2p}(\R)}\right)^{1-\frac{1}{2p-1}} \\
	&\lesssim \|u(t)-u_M\|_{L^1(\R)}^{\frac{1}{2p-1}}\left(t^{-\frac12(1-\frac{1}{2p})}\right)^{1-\frac{1}{2p-1}} = o(t^{-\frac12(1-\frac{1}{p})}).
\end{align*}
The proof of Theorem \ref{teo:asymptotic} is now complete.

%%%%%%%%%%%% SECTION 4 %%%%%%%%%%%%%%%%%%%%%%%%%

\section{Numerical example}\label{sec4}
To conclude this paper, we show some numerical simulations that illustrate the analytical results that we have proved in Section \ref{sec2} and Section \ref{sec3}. Let us remark that the use of splitting methods facilitates the numerical resolution of the different terms of an equation independently. In our case, one can deal with the nonlocal term more efficiently while using appropriate schemes for the nonlinear phenomena arising due to the flux term.

The solution to the augmented Burgers equation, defined as in \eqref{eq:abe}, is immediately dominated by the viscous effects of the terms in the right-hand side. In order to emphasize the effects of the nonlinear flux, which are predominant in the sonic-boom phenomenon, let us take some non-unitary values of the parameters in equation \eqref{eq:sonic} to perform the simulations:
\begin{equation*}
		u_t-\left(\frac{u^2}{2}\right)_x=\frac{1}{\Gamma}u_{xx}+c_\nu(K*u-u+u_x).
\end{equation*}
In particular, let us  consider $\Gamma=100$ and $c_\nu=0.02$ in our experiments. Note that the analytical results obtained in the previous sections, regarding convergence order and decay rates, are still satisfied.

For some grid parameters $\dx,\dt>0$, let us denote by $u^n_j$ the approximate solution to $u(n\dt,j\dx)$, with $n\in\mathbb{N}\cup\{0\}$ and $j\in\Z$. On the one hand, in order to solve equation \eqref{eq:burgers}, we use Engquist-Osher numerical flux and centered finite differences for the diffusion:
\begin{equation}\label{eq:scheme}
	u^{n+\frac12}_j=u_n^j-\frac{\Delta t}{\Delta x}\left(g(u^n_{j},u^n_{j+1})-g(u^n_{j-1},u^n_{j})\right) + \frac{1}{\Gamma}\frac{\dt}{\dx^2}\left(u^n_{j-1}-2u^n_{j}+u^n_{j+1}\right),
\end{equation}
where 
\begin{equation*}
	g(a,b)=-\frac{a(a-|a|)}{4}-\frac{b(b+|b|)}{4},\quad \forall a,b\in\rr.
\end{equation*}
This choice for the flux is in accordance with the large-time performance requirements listed in \cite{Pozo:2014,Pozo:2015}. 

On the other hand, regarding the equation corresponding to the nonlocal term \eqref{eq:augmented}, note that it can be rewritten in the following way:
\begin{equation*}
	\begin{cases}
		v_t+v_{tx}=c_\nu v_{xx},&(t,x)\in(0,\infty)\times\R, \\
		v(0,x)=v_0(x),&x\in\R.
	\end{cases}
\end{equation*}
This form allows us to use centered finite differences along with Crank-Nicolson (e.g. \cite{Iserles:1996}):
\begin{flalign*}
	\frac{u^{n+1}_j-u^{n+\frac12}_j}{\dt} &+ \left(\frac{u^{n+1}_{j+1}-u^{n+\frac12}_{j+1}}{\dx\dt} -\frac{u^{n+1}_{j-1}-u^{n+\frac12}_{j-1}}{\dx\dt}\right) \\ &=\frac{c_\nu}{2}\left( \frac{u^{n+1}_{j-1}-2u^{n+1}_{j}+u^{n+1}_{j+1}}{\dx^2} + \frac{u^{n+\frac12}_{j-1}-2u^{n+\frac12}_{j}+u^{n+\frac12}_{j+1}}{\dx^2}\right) ,\,\,\, j\in \mathbb{Z}, n>0.
\end{flalign*}
This alternative has already been used in \cite{Cleveland:1995,Rallabhandi:2011a,Rallabhandi:2011b}. Any tridiagonal matrix solver can efficiently solve it, avoiding the computational cost of the integration of the convolution in \eqref{eq:augmented}. Additionally, simulations show that, combined with \eqref{eq:burgers}, the long-time dynamics of \eqref{eq:abe} are preserved, while the solution keeps the scale given by \eqref{eq:ul}.

We take a mesh size $\dx=0.1$ and a time step small enough to satisfy the stability condition (e.g., see \cite{Castro:2010}) for scheme \eqref{eq:scheme}, given by:
\begin{equation*}
	\left(\max_{j\in\Z} \left|u^0_j\right|\right)^2 \frac{\dt}{\dx}+\frac{2}{\Gamma}\frac{\dt}{\dx^2} \le 1.
\end{equation*}
In order to avoid boundary issues, we choose compactly supported initial data and a large enough spatial domain.

First of all, we focus on the convergence order of the splitting method. Since we do not know the exact solution of \eqref{eq:abe}, we determine the order by comparing two numerical solutions for the same initial data but a different time step. That is, we compute $\|u_1(T)-u_2(T)\|_{L^2(\R)}$, where $u_1$ and $u_2$ are obtained from the same initial data using $\dt$ and $\dt/2$ time-steps, respectively. As a matter of fact, we show the convergence order for two different initial data, plotted in Figure \ref{fig:fig4}. 

\begin{figure}[ht!]
	\centering
	\includegraphics[width=0.8\linewidth]{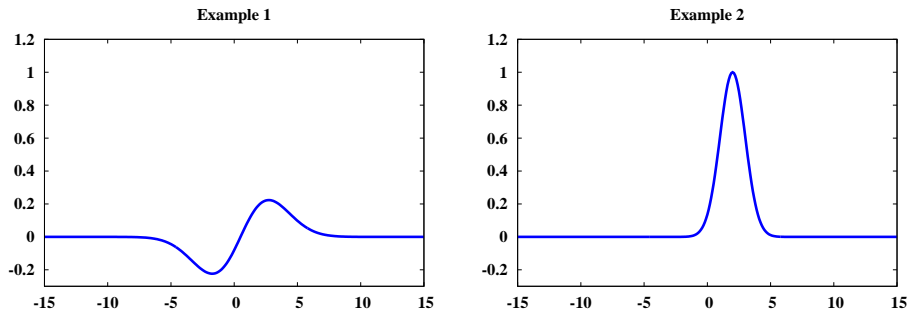}
	\caption{Initial data used in the first experiment, in relation to the accuracy order.}
	\label{fig:fig4}
\end{figure}

As it can be observed in Figure \ref{fig:fig1}, the slopes for the numerical solutions at $T=10$ corresponding to those initial data are consistent with the theoretical results of Section \ref{sec2}. A line of slope one has been added as a reference for the convergence order.

\begin{figure}[ht!]
	\centering
	\includegraphics[width=0.7\linewidth]{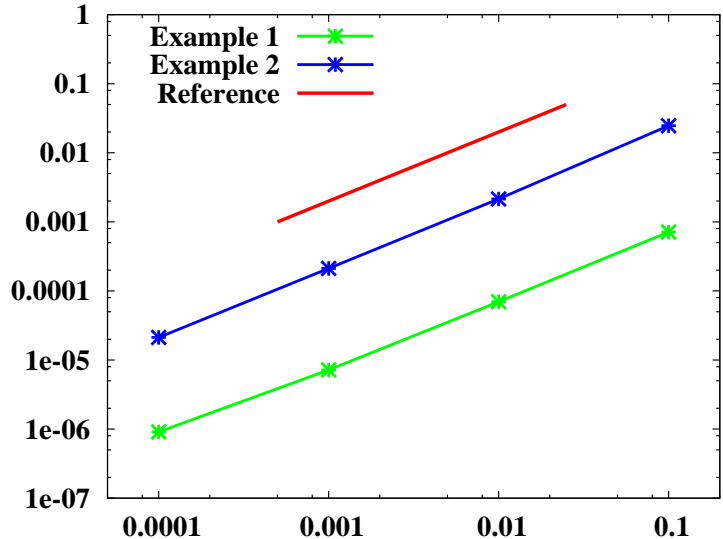}
	\caption{Accuracy order for the initial data shown in Figure \ref{fig:fig4} using several time-step sizes. The plot shows $\|u_1(10)-u_2(10)\|_{L^2(\R)}$, where $u_1$ and $u_2$ have been computed using $\dt$ and $\dt/2$ respectively.}
	\label{fig:fig1}
\end{figure}

Furthermore, we also show the convergence rate of the numerical solutions towards the asymptotic profile of the augmented Burgers equation, given in Theorem \ref{teo:asymptotic}. Let us remark that $u_M$ is now the asymptotic profile of the viscous Burgers equation with $1/\Gamma+c_\nu$ viscosity coefficient instead of $2$ (see \cite{Pozo:2014} for the complete expression for $u_M$). We compare the splitting method described above with the discretization proposed in \cite{Pozo:2014}, where a rectangle method is used to approximate the convolution in \eqref{eq:abe}. In Figure \ref{fig:fig2} we plot the solution at time $T=12000$ for the given initial data, as well as the corresponding asymptotic profile $u_M$. Then, in Figure \ref{fig:fig3} we can observe that the convergence rates meet the ones defined in \eqref{eq:convrates}, preserving the large-time behavior of the solution of equation \eqref{eq:abe}.

\begin{figure}[ht!]
	\centering
	\includegraphics[width=0.95\linewidth]{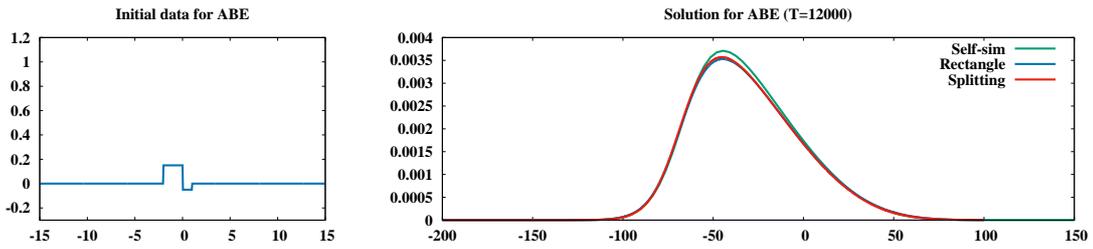}
	\caption{Initial data and solution for ABE with $\Gamma=100$ and $c_\nu=0.02$ at $T=12000$. We compare the splitting method (blue) with the one from \cite{Pozo:2014} discretized explicitly in time (red). The self-similar asymptotic profile (green) is given as reference too.}
	\label{fig:fig2}
\end{figure}

\begin{figure}[ht!]
	\centering
	\includegraphics[width=0.95\linewidth]{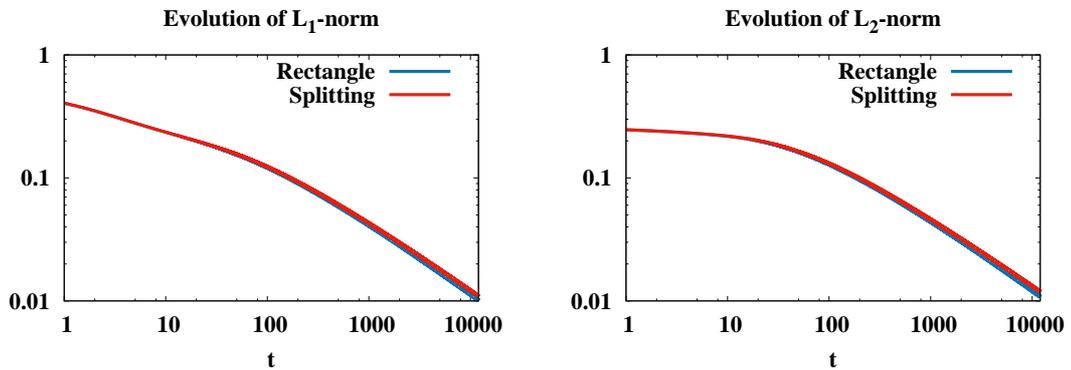}
	\caption{Evolution of the $L^1$ and $L^2$ norms of the difference between the asymptotic profile and the solution, multiplied by their corresponding rate, as in \eqref{eq:convrates}. We compare the splitting method (blue) and the one from \cite{Pozo:2014} discretized explicitly in time (red).}
	\label{fig:fig3}
\end{figure}

\section*{Aknowledgements}
L.~I.~Ignat was partially supported by a grant of the Romanian National Authority for Scientific Research and Innovation, CNCS-UEFISCDI, project number  PN-II-RU-TE-2014-4-0007, Grant MTM2014-52347, MICINN, Spain and FA9550-15-1-0027 of AFOSR. A.~Pozo was granted by the Basque Government, reference PRE\_2013\_2\_150, and partially supported by ERCEA under Grant 246775 NUMERIWAVES, by the Basque Government through the BERC~2014-2017 program and by Spanish MINECO:~BCAM~Severo~Ochoa excellence accreditation SEV-2013-0323.

%Part of this paper was done during the visit of the second author to the Institute of Mathematics ``Simion Stoilow'' of the Romanian Academy. He thanks the Institute for hospitality and support. 

\bibliographystyle{amsplain}
\bibliography{library}

\end{document}